\documentclass[11pt]{amsart}
\usepackage{amsxtra,amssymb,amsmath,amscd,url,listings,stmaryrd}
\usepackage[alphabetic,initials]{amsrefs}
\usepackage{cases}
\usepackage{mathrsfs}
\usepackage{bbm}
\usepackage{amssymb}
\usepackage{amscd}
\usepackage{amsfonts,latexsym,amsthm,amsxtra,mathdots,amssymb,latexsym}
\usepackage[all,cmtip]{xy}
\RequirePackage{amsmath} \RequirePackage{amssymb}
\usepackage{color}
\usepackage{colordvi}
\usepackage{multicol}
\usepackage{hyperref}
\usepackage{mathtools}
\usepackage[margin=1in]{geometry}
\usepackage{xcolor}
\hypersetup{
    colorlinks,
    linkcolor={red!50!black},
    citecolor={blue!50!black},
    urlcolor={blue!80!black}
}

\newtheorem{conjecture}{Conjecture}
\newtheorem{problem}{Problem}

\newtheorem{theorem}{Theorem}

\newtheorem*{proposition*}{Proposition}
\newtheorem{lemma}{Lemma}
\newtheorem{corollary}{Corollary}

\newcommand{\N}{{\mathbb N}}

\newcommand{\RR}{\mathscr{R}}

\newcommand{\lb}{\left(}
\newcommand{\rb}{\right)}
\def \i {{\rm i}}
\def \d {{\rm d}}
\def \e {{\rm e}}

\newcommand{\D}{{\mathcal D}}
\newcommand{\CC}{{\mathcal C}}

\newcommand{\R}{{\mathbb R}}

\newcommand{\A}{\mathbb{A}}
\def \a {\mathbf{a}}

\newcommand{\Er}{{\mathscr{E}}}

\begin{document}
\title{Mean values of ratios of the Riemann zeta function}
 \author[Daodao Yang]{Daodao Yang}

\address{Institute of Analysis and Number Theory \\ Graz University of Technology \\ Kopernikusgasse 24/II, 
A-8010 Graz \\ Austria}

\address{Current address: D\'epartement de math\'ematiques et de statistique\\
		Universit\'e de Montr\'eal\\
		CP 6128 succ. Centre-Ville\\
		Montr\'eal, QC H3C 3J7\\
		Canada}

\email{yangdao2@126.com  \quad dyang@msri.org}

\begin{abstract}
       It is proved that  
      $$\int_{T}^{2T} \left|\frac{\zeta\left(\frac{1}{2}+{\rm i} t\right)}{\zeta\left(1+2{\rm i} t\right)}\right|^2 {\rm d} t = \frac{1}{\zeta(2)} T \log T + \left( \frac{\log \frac{2}{\pi} + 2\gamma -1 }{\zeta(2)}  -4 \,\frac{\zeta^{\prime}(2)}{\zeta^2(2)}  \right) T  +  O\left(T\, \left(\log T\right)^{-2023} \right) , \quad \forall T \geqslant 100.   $$
      For given  $a\in \N$, we also establish similar formulas for second moments of $|\zeta(\tfrac{1}{2} + {\rm i} t)/\zeta(1 + {\rm i} at)|.$ We have \begin{align*}
          \lim_{a \to \infty} \lim_{T \to \infty}\frac{1}{T \log T} \int_{T}^{2T} \left|\frac{\zeta\left(\frac{1}{2}+{\rm i} t\right)}{\zeta\left(1+{\rm i} at\right)}\right|^2 {\rm d} t = \frac{\zeta(2)}{\zeta(4)}.
      \end{align*}
\end{abstract}

\maketitle

 \section{Introduction}

This paper establishes the following new results for mean values of $|\zeta(\tfrac{1}{2} + \i t)/\zeta(1 + \i at)|.$

 \begin{theorem}\label{Main}
 Fix $a\in \N$ and  $B > 0$. Let $T \geqslant 100$, then
 \begin{align*}
\int_{T}^{2T} \left|\frac{\zeta\lb\frac{1}{2}+\i t\rb}{\zeta\lb1+\i at\rb}\right|^2 \d t = \D_1(a)  T \log T + \D_0(a)  T   + O_{a,\,B}\lb T\, \lb \log T\rb^{-B} \rb ,
\end{align*}
 
 where

\begin{align*}
    \D_1(a)&\, =  \prod_{p} \lb1 - \frac{2}{p^{1+\frac{a}{2}}} + \frac{1}{p^2} \rb ,\\
    \D_0(a)&\, =  \lb \log \frac{2}{\pi} + 2\gamma -1 \rb \D_1(a)  +   \widetilde{\D}_0(a) , \\
    \widetilde{\D}_0(a)&\, = 2a \prod_{p} \lb1 - \frac{2}{p^{1+\frac{a}{2}}} + \frac{1}{p^2} \rb \cdot \sum_{p}^{}\frac{\log p}{ p^{1+\frac{a}{2}}} \lb1 - \frac{2}{p^{1+\frac{a}{2}}} + \frac{1}{p^2}\rb^{-1} .
\end{align*}
 \end{theorem}

When $a =  2$, we have the following simple formula.
 \begin{corollary}\label{Cor1}
  \begin{align*}
\int_{T}^{2T} \left|\frac{\zeta\left(\frac{1}{2}+\i t\right)}{\zeta\left(1+2\i t\right)}\right|^2 \d t = \frac{1}{\zeta(2)} T \log T + \lb \frac{\log \frac{2}{\pi} + 2\gamma -1 }{\zeta(2)}  -4 \,\frac{\zeta^{\prime}(2)}{\zeta^2(2)}  \rb T + O\left(T\, \left(\log T\right)^{-2023} \right) , \quad \forall T \geqslant 100. 
\end{align*}
  \end{corollary}
  
It is easy to see that $\D_1(a)$ is a strict increasing function of $a$ and the limit $\lim_{a \to \infty} \D_1(a)$ exists.
\begin{corollary}\label{LimiOfa} Let $a \in \N$, then
    \begin{align*}
          \lim_{a \to \infty} \lim_{T \to \infty}\frac{1}{T \log T} \int_{T}^{2T} \left|\frac{\zeta\lb\frac{1}{2}+\i t\rb}{\zeta\lb1+\i at\rb}\right|^2 \d t = \frac{\zeta(2)}{\zeta(4)}.
      \end{align*}
\end{corollary}

The study of  mean values  of $|\zeta(\sigma + \i t)|$  has a long history, tracing back to the work of  Landau \cite{Landau} and Schnee \cite{Sch}. In 1909, they established that for fixed $\sigma > 1/2$, we have
\begin{align}\label{LanSch}
     \int_{T}^{2T} \left|\zeta\lb\sigma+\i t\rb\right|^2 \d t \sim \zeta(2\sigma)T, \quad \text{as} \quad T \to \infty. 
\end{align}
Later in 1916,  Hardy and Littlewood \cite{HaLw} established that 
\begin{align}\label{HL}
     \int_{T}^{2T} \left|\zeta\lb\frac{1}{2}+\i t\rb\right|^2 \d t \sim T \log T, \quad \text{as} \quad T \to \infty. 
\end{align}
The asymptotic formula \eqref{HL}  was refined by Ingham \cite{Ingham} who established\footnote{The error term $O(T^{\frac{1}{2}}\log T)$ in \eqref{Ingham} has been refined several times. See \cites{Titchmarsh, Bala, HH, Huxley,  BW}.} that

\begin{align}\label{Ingham}
     \int_{T}^{2T} \left|\zeta\lb\frac{1}{2}+\i t\rb\right|^2 \d t = T \log T +  \lb  \log \frac{2}{\pi}+ 2\gamma - 1 \rb T + O\lb T^{\frac{1}{2}}\log T\rb.
\end{align}

 Formula \eqref{LanSch} tells us that  the mean value of $|\zeta(1+2\i t)|^2$  is $\zeta(2)$ and formula \eqref{HL} shows that the mean value of $|\zeta(1/2 + \i t)|^2$ on  $[T, 2T]$  is $\log T$. Thus, our formula in Corollary \ref{Cor1} could be viewed as a mixture of the two formulas. Roughly speaking,  Corollary \ref{Cor1} tells us that \emph{the mean value of the ratios equals the ratios of the mean values}. Namely,  we have
\begin{align}\label{special}
 \frac{1}{T}   \int_{T}^{2T} \left|\frac{\zeta\left(\frac{1}{2}+\i t\right)}{\zeta\left(1+2\i t\right)}\right|^2 \d t \sim \frac{\frac{1}{T}\int_T^{2T}\left|\zeta\lb\frac{1}{2}+\i t\rb\right|^2 \d t}{\frac{1}{T}\int_T^{2T} |\zeta(1+ 2 \i t)|^2 \d t}, \quad \text{as} \quad T \to \infty.
\end{align}

Furthermore, if we divide both sides of formula \eqref{Ingham} by $\zeta(2)$, we almost recover the formula in Corollary \ref{Cor1}, except for the second main term, which differs by $-4\zeta^{\prime}(2)/\zeta^2(2)T$. On the other hand, we note that such an interesting formula \eqref{special} does not exist for $|\zeta(\tfrac{1}{2} + \i t)/\zeta(1 + \i at)|^2$ when $a \neq 2$. This might deserve further research and explanation: why is $a = 2$ so special?

Our motivation for studying mean values of $|\zeta(\tfrac{1}{2} + \i t)/\zeta(1 + \i at)|^2$ comes from  spectral theory of automorphic forms.    When using Kuznetsov’s spectral reciprocity formula (see \cites{Kuz89, Kuz99, Mot03} and \cite{PHK}*{Theorem 3.1}), mean values of  $|\zeta(\tfrac{1}{2} + \i t)^k/\zeta(1 + \i at)^j|$ may appear. For instance, in \cite{PHK}*{Lemma 5.1.}, the authors need to calculate the mean value of $|\zeta(\tfrac{1}{2} + \i t)^4/\zeta(1 + 2\i t)|^2$. Even though the authors only need upper bounds for $\int_T^{2T}|\zeta(\tfrac{1}{2} + \i t)^4/\zeta(1 + 2\i t)|^2\d t$, it still might be interesting to   establish the following asymptotic formula 
\begin{equation*}
\int_{T}^{2T} \left|\frac{\zeta\left(\frac{1}{2}+\i t\right)^4}{\zeta\left(1+\i at\right)}\right|^2 dt \sim \D(a) \, T(\log T)^{16}\,,\quad \text{as} \quad T \to \infty,
\end{equation*}
and find for a given $a \in \R_{>0}$ the positive constant $\D(a)$. It seems unlikely to prove the above formula in the near future,  so we consider the easier problem in this paper, i.e., the second moment of $|\zeta(\tfrac{1}{2} + \i t)/\zeta(1 + \i at)|.$ Our result is the first example of mean values of the type $|\zeta(\tfrac{1}{2} + \i t)^k/\zeta(1 + \i at)^j|$ and may provide new directions for research on the value distributions of zeta and  $L$-functions.  

On the other hand,  we mention that in \cite{Farmer}, Farmer  studied the following mean values  $$\RR(\alpha, \beta, \gamma, \delta): = \int_0^T \frac{ \zeta(\frac{1}{2} + \i t + \alpha)\zeta(\frac{1}{2} - \i t + \beta)}{\zeta(\frac{1}{2} + \i t + \gamma) \zeta(\frac{1}{2} - \i t + \delta)} \d t\,.$$ 

Furthermore, Farmer \cite{Farmer}*{(7.4)} made the  \emph{ratios conjecture} which states that when $\alpha, \beta, \gamma, \delta \ll 1/ \log T$ and $\Re(\gamma), \Re(\delta) > 0,$ then
\begin{align*}
    \RR(\alpha, \beta, \gamma, \delta)  \sim T \frac{(\alpha + \delta)(\beta + \gamma)}{(\alpha + \beta)(\gamma + \delta)} - T^{1-\alpha-\beta}\frac{(\alpha-\gamma)(\beta-\delta)}{(\alpha + \beta)(\gamma + \delta)},\quad \text{as} \quad T \to \infty.
\end{align*}

Farmer's  conjecture concerns ratios of zeta functions near the critical line, while our result concerns ratios of zeta functions along different fixed vertical lines.  Farmer's ratios conjecture has been generalized to other families of $L$-functions. For further background, see \cites{CFZ, CS}.

\section{Proof outline}  We first make use of the  zero-free region
of Vinogradov-Korobov \cites{VinZ, KorZ} to approximate $1/\zeta(1+\i at)$ by a short Dirichlet polynomial:
\[ \frac{1}{\zeta(1+\i at)} \approx  \sum_{n \leqslant  X }  \frac{\mu(n)}{n^{1+ \i at}},  \quad T \leqslant t \leqslant 2T,  \]
where the length $X$ is much smaller than any power of $T$. So to establish Theorem \ref{Main},  we need to  compute the mean value of the product of the Riemann zeta function and the above  Dirichlet polynomial: 
\[ \int_T^{2T}\left|\zeta\lb\frac{1}{2}+\i t\rb\right|^2\left| \sum_{n \leqslant  X }  \frac{\mu(n)}{n^{1+ \i at}}\right|^2 \d t.\]
The tool to  tackle this problem is a theorem by Balasubramanian, Conrey and Heath-Brown  \cite{BCH}. According to their theorem, it suffices to compute the following two types of sums (where $b = 1 -a/2$):
\[\sum_{m,\,n \leqslant X 
 }\frac{\mu(m) \cdot \mu(n)}{(mn)^b [m^a, n^a]}, \quad \quad \quad \quad \sum_{m,\,n \leqslant X 
 }\frac{\mu(m) \cdot \mu(n)}{(mn)^b [m^a, n^a]} \log\frac{(m^a, n^a)}{[m^a, n^a]}.\]
These sums could be considered as variants of GCD sums and log-type GCD sums (see \cites{G, DY, logGCD}).   The main difference is that, in our case, we have the extra M\"obius function as weights, which makes our problem much eaiser than the problems in \cites{G, DY, logGCD}. In our case, only square-free integers  appear in the sums, and the sums converge as $X \to \infty$.  To calculate the limits, we  derive two new identities  inspired by  G\'{a}l's identity\footnote{Refer to equation (34) in \cite{G} for the original version and equation (13) in \cite{logGCD} for a generalized version.},  which transforms a sum into a product.  To estimate the error terms, we use Rankin’s trick to convert them into multiplicative objects that are easier to handle.
 \section{Lemmas}

In this section, we present several lemmas that will be used in the proof of  Theorem \ref{Main}.  The main tool is the following theorem established by Balasubramanian, Conrey and Heath-Brown \cite{BCH}.

\begin{lemma}[Balasubramanian--Conrey--Heath-Brown  \cite{BCH}]\label{BCHB}
  Let $A(s) = \sum_{m \leqslant M} a(m) m^{-s}$ . Suppose that $a(m) \ll_{\epsilon} m^{\epsilon}$ for any $\epsilon > 0$ and $\log M \ll \log T$. Then
  
  \begin{align*}
      \int_{0}^{T} \left|\zeta(\frac{1}{2}+\i t)\right|^2 \left|A\lb\frac{1}{2} + \i t\rb\right|^2 \d t = T \sum_{h,\, k\, \leqslant M} \frac{a(h)}{h}\frac{\overline{a(k)}  }{k} (h, k)\left( \log \frac{T(h, k)^2}{2 \pi h k} + 2\gamma -1\right) + \Er,
  \end{align*}

where $\Er = \Er_1 + \Er_2 $ with $\Er_1 
 \ll_B T \lb \log T\rb^{-B} $ for any $B > 0$ and $\Er_2 \ll_{\epsilon} M^2 T^{\epsilon}$ for any $\epsilon > 0$.
  
\end{lemma}

Second, we have the following lemma, which is useful when approximating $1/\zeta(1  + \i t)$ by Dirichlet polynomials. The proof follows the proof of Theorem 3.13 in \cite{Tit}.

\begin{lemma}\label{approx}
 Let $a \in (0, \infty)$, $\beta \in (\frac{2}{3}, 1) $ and $\eta \in (0, \beta - \frac{2}{3})$ be fixed.  For  $T \geqslant 100$, the following inequality
    \begin{align*}
    \left| \frac{1}{\zeta(1+\i t)}  - \sum_{n \leqslant  \e^{\left(\log T\right)^{\beta}} }   \frac{\mu(n)}{n^{1+ \i t}} \right| \ll \e^{-\left(\log T\right)^{\eta}}\,, %O\left(  \e^{-\left(\log T\right)^{\eta}} \right)
        \end{align*}
     holds   uniformly for all $aT \leqslant t \leqslant 2aT$, where the implied constant  only depends on $a$, $\beta$ and $\eta$.
 
\end{lemma}

\begin{proof}
We present the proof for the case $a = 1$. The proof for general $a$ is the same. Let $ c = \frac{1}{\log x}$ and $x =  \e^{\left(\log T\right)^{\beta}}$. For $T\geqslant 100$, we use an effective version of Perron's formula (see \cite{koukou}*{Theorem 7.2}) to obtain the following 
\begin{align*}
 \sum_{n \leqslant x}    \frac{\mu(n)}{n^{1+ \i t}}  =  \frac{1}{2\pi \i} \int_{c - \i \frac{T}{2}}^{c + \i \frac{T}{2}} 
\frac{1}{\zeta( 1 + \i t +s)} \frac{x^s}{s} \d s  + O \lb \frac{\log x}{T} + \frac{1}{x}\rb .
\end{align*}

By   the  zero-free region
of Vinogradov-Korobov \cites{VinZ, KorZ} and Theorem 3.11 in \cite{Tit}, we find that there exists a positive constant $A$ such that
\begin{align}\label{VK: estimate}
    \left| \frac{1}{\zeta(\sigma + \i t)} \right| \ll (\log T)^{\frac{2}{3}}(\log \log T)^{\frac{1}{3}}, \quad \forall  \sigma \geqslant 1 -  \frac{A}{(\log T)^{\frac{2}{3}}(\log \log T)^{\frac{1}{3}} } \quad \text{and} \quad \frac{T}{2} \leqslant t \leqslant \frac{5T}{2}\,.
\end{align}

Take $\delta = \frac{A}{(\log T)^{\frac{2}{3}}(\log \log T)^{\frac{1}{3}} }$. When $T \leqslant t \leqslant 2T$, $|\Im(s)| \leqslant \frac{T}{2}$, and $\Re(s) \geqslant -\delta$, we have $\frac{T}{2} \leqslant \Im(s+\i t) \leqslant \frac{5T}{2}$ and $\Re(1 + s) \geqslant 1 - \delta$.
By the residue theorem,  we get
\begin{align*}
  \frac{1}{2\pi \i} \int_{c - \i \frac{T}{2}}^{c + \i \frac{T}{2}} 
\frac{1}{\zeta( 1 + \i t +s)} \frac{x^s}{s} \d s  = \frac{1}{\zeta(1  + \i t)}  + \frac{1}{2 \pi \i} \lb \int_{ c - \i  \frac{T}{2}}^{-\delta - \i \frac{T}{2}} +  \int_{-\delta - \i \frac{T}{2}}^{-\delta + \i \frac{T}{2}}   +  \int_{-\delta + \i \frac{T}{2}}^{c + \i \frac{T}{2}} \rb.
\end{align*}
 Next, we will use \eqref{VK: estimate} to estimate the integrals on the right side of the above formula. We have

\begin{align*}
    \left|  \int_{-\delta - \i \frac{T}{2}}^{-\delta + \i \frac{T}{2}} 
\frac{1}{\zeta( 1 + \i t +s)} \frac{x^s}{s} \right| &\ll (\log T)^{\frac{2}{3}}(\log \log T)^{\frac{1}{3}}   x^{-\delta} \int_{-\frac{T}{2}}^{\frac{T}{2}} \frac{\d y}{\sqrt{\delta^2 + y^2}}\\
&\ll (\log T)^{\frac{5}{3}}(\log \log T)^{\frac{1}{3}} \frac{1}{\exp \lb \frac{A}{(\log T)^{\frac{2}{3}}(\log \log T)^{\frac{1}{3}}} (\log T)^{\beta}\rb} \\
&\ll  \e^{-\left(\log T\right)^{\eta}}\,,
\end{align*}

 and

\begin{align*}
     \left|  \int_{-\delta + \i \frac{T}{2}}^{c + \i \frac{T}{2}} 
\frac{1}{\zeta( 1 + \i t +s)} \frac{x^s}{s} \right| &\ll \frac{(\log T)^{\frac{2}{3}}(\log \log T)^{\frac{1}{3}}}{T} \frac{1}{\log x}\lb x^c - x^{-\delta} \rb \ll \frac{1}{T}.
\end{align*}

We could estimate the integral $\int_{ c - \i  \frac{T}{2}}^{-\delta - \i \frac{T}{2}}$ in a similar way. Then we finish the proof of the lemma.
    \end{proof}

The following four lemmas are intended for calculating the coefficients in the asymptotic formula.

\begin{lemma}\label{Constan1}
    Let $a\in \N$ and $b = 1-\frac{a}{2}$, then
    $$\sum_{m=1}^{\infty}\sum_{n=1}^{\infty} \frac{\mu(m) \cdot \mu(n)}{(mn)^b [m^a, n^a]}  = \prod_{p} \big(1 - \frac{2}{p^{1+\frac{a}{2}}} + \frac{1}{p^2} \big) \,.$$
    In particular, when $a=2$, one has
    $$\sum_{m=1}^{\infty}\sum_{n=1}^{\infty} \frac{\mu(m) \cdot \mu(n)}{ [m^2, n^2]}  = \prod_{p} \big(1 -  \frac{1}{p^2} \big) = \frac{1}{\zeta(2)} \,.$$
\end{lemma}

\begin{proof} Note that the sum is over square-free integers, and by multiplicity, we can rewrite the sum as a product over prime numbers
    \begin{align*}
\sum_{m=1}^{\infty}\sum_{n=1}^{\infty} \frac{\mu(m) \cdot \mu(n)}{(mn)^b [m^a, n^a]}  = &\prod_{i = 1} ^{\infty}\lb \sum_{\alpha_i, \beta_i \in \{0,\, 1 \}} \frac{(-1)^{\alpha_i + \beta_i}}{p_i^{a \max \{ \alpha_i, \beta_i\} + b(\alpha_i + \beta_i) }} \rb\\
= & \prod_{i = 1} ^{\infty} \lb 1 - \frac{2}{p_i^{a + b}} + \frac{1}{p_i^{a + 2b}}\rb
\\
= & \prod_{i = 1} ^{\infty} \lb 1 - \frac{2}{p_i^{1+\frac{a}{2}}} + \frac{1}{p_i^{2}}\rb.
    \end{align*}
    
\end{proof}

\begin{lemma}\label{Constan1: err}Let $a\in \N$ and $b = 1-\frac{a}{2}$, then
$$
  \left|   \sum_{\substack{ m, n \in \N \\ m > X \text{ \emph{ or}}\,\; n > X}}   \frac{\mu(m) \cdot \mu(n)}{(mn)^b [m^a, n^a]}  \right|    \ll\begin{cases}
			\frac{\log X}{\sqrt X} , & \text{if\, $a = 1$ }\\
			\frac{(\log X)^2}{X} , & \text{if\, $a = 2$ }\\
            \frac{\log X}{X}, & \text{if\, $a \geqslant 3 $.}
		 \end{cases}
$$

\end{lemma}

\begin{proof}  
 By  Rankin’s trick, 
\begin{align*}
 \left|  \sum_{\substack{ m, n \in \N \\ m > X \,\;\text{or}\,\; n > X}}  \frac{\mu(m) \cdot \mu(n)}{(mn)^b [m^a, n^a]}  \right| & \leqslant 2    \sum_{\substack{ m, n \in \N \\ m > X }}  \frac{\mu^2(m) \cdot \mu^2(n)}{(mn)^b [m^a, n^a]}\\
 & \leqslant \frac{2}{X^{1-\epsilon}}    \sum_{\substack{ m, n \in \N  }}  \frac{\mu^2(m) \cdot \mu^2(n)}{(mn)^b [m^a, n^a]}m^{1-\epsilon}\\
 & = \frac{2}{X^{1-\epsilon}} \prod_{i = 1}^{\infty}\sum_{\alpha_i = 0}^{1}\sum_{\beta_i = 0}^{1}\frac{1}{p_i^{ b(\alpha_i+\beta_i) + a\,\cdot \max\{\alpha_i,\beta_i\} - (1-\epsilon) \alpha_i  }}\\
 &=\frac{2}{X^{1-\epsilon}} \prod_{p} \bigg( 1+\frac{1}{p^{\frac{a}{2}+\epsilon} }  + \frac{1}{p^{1+\frac{a}{2}}}  + \frac{1}{p^{1+\epsilon}}  \bigg).
 \end{align*}

Clearly, the above product over primes is $\ll \epsilon^{-2}$ when $ a= 2$, and is $\ll \epsilon^{-1}$ when $ a \geqslant 3$. Take $\epsilon = 1/ \log X$, then we obtain the inequality.

When $a =1$, similarly using Rankin’s trick, we obtain
\begin{align*}
 \left|  \sum_{\substack{ m, n \in \N \\ m > X \,\;\text{or}\,\; n > X}}  \frac{\mu(m) \cdot \mu(n)}{\sqrt{mn}\, [m, n]}  \right| & \leqslant \frac{2}{X^{\frac{1}{2}-\epsilon}} \prod_{i = 1}^{\infty}\sum_{\alpha_i = 0}^{1}\sum_{\beta_i = 0}^{1}\frac{1}{p_i^{   \epsilon \alpha_i +  \frac{1}{2}\beta_i + \max\{\alpha_i, \beta_i\}  }}\\
 &=\frac{2}{X^{\frac{1}{2}-\epsilon}} \prod_{p} \bigg( 1+\frac{1}{p^{1+\epsilon} }  + \frac{1}{p^{\frac{3}{2}}}  + \frac{1}{p^{\frac{3}{2}+\epsilon}}  \bigg)\\&\ll \frac{1}{X^{\frac{1}{2}-\epsilon}}\, \epsilon^{-1}.
 \end{align*}
 
Again, take $\epsilon = 1/ \log X$, then we are done.
\end{proof}

\begin{lemma}\label{Constan2: err}Let $a\in \N$ and $b = 1-\frac{a}{2}$, then
$$
  \left|   \sum_{\substack{ m, n \in \N \\ m > X \text{ \emph{ or}}\,\; n > X}}   \frac{\mu(m) \cdot \mu(n)}{(mn)^b [m^a, n^a]} \log\frac{(m^a, n^a)}{[m^a, n^a]} \right|    \ll\begin{cases}
			\frac{(\log X)^2}{\sqrt X} , & \text{if\, $a = 1$ }\\
			\frac{(\log X)^3}{X} , & \text{if\, $a = 2$ }\\
            \frac{(\log X)^2}{X}, & \text{if\, $a \geqslant 3 $.}
		 \end{cases}
$$

\end{lemma}

\begin{proof}
Note that $\forall \delta > 0, \forall m, n \in \N$, we have
$$ \left| \log\frac{(m^a, n^a)}{[m^a, n^a]}\right| \leqslant 2a\, \log[m, n] \leqslant 2a \cdot \frac{1}{\delta} [m, n]^{\delta}.$$
Following the proof  in the previous lemma, when $a \geqslant 2$, we have
\begin{align*}
     \left|   \sum_{\substack{ m, n \in \N \\ m > X \text{ \emph{ or}}\,\; n > X}}   \frac{\mu(m) \cdot \mu(n)}{(mn)^b [m^a, n^a]} \log\frac{(m^a, n^a)}{[m^a, n^a]} \right|  & \ll \frac{1}{\delta X^{1-\epsilon}} \prod_{i = 1}^{\infty}\sum_{\alpha_i = 0}^{1}\sum_{\beta_i = 0}^{1}\frac{1}{p_i^{ b(\alpha_i+\beta_i) + (a-\delta)\,\cdot \max\{\alpha_i,\beta_i\} - (1-\epsilon) \alpha_i  }}\\
& = \frac{1}{\delta X^{1-\epsilon}} \prod_{p} \bigg( 1+\frac{1}{p^{\frac{a}{2}-\delta+\epsilon} }  + \frac{1}{p^{1+\frac{a}{2}-\delta}}  + \frac{1}{p^{1-\delta+\epsilon}}  \bigg).
\end{align*}

To obain the estimate, we take $\epsilon = 1/ \log X$ and $\delta = \epsilon/2$. Similar process for $ a = 1 $.

\end{proof}

\begin{lemma}\label{Constan2} Let $a\in \N$ and $b = 1-\frac{a}{2}$, then
    $$\sum_{m=1}^{\infty}\sum_{n=1}^{\infty} \frac{\mu(m) \cdot \mu(n)}{(mn)^b [m^a, n^a]} \log\frac{(m^a, n^a)}{[m^a, n^a]} = 2a \prod_{p} \big(1 - \frac{2}{p^{1+\frac{a}{2}}} + \frac{1}{p^2} \big) \cdot \sum_{p}^{}\frac{\log p}{ p^{1+\frac{a}{2}}} (1 - \frac{2}{p^{1+\frac{a}{2}}} + \frac{1}{p^2})^{-1}\,.$$
        In particular, when $a=2$, one has
$$\sum_{m=1}^{\infty}\sum_{n=1}^{\infty} \frac{\mu(m) \cdot \mu(n)}{ [m^2, n^2]} \log\frac{(m^2, n^2)}{[m^2, n^2]} = 4\prod_{p} \big(1 -  \frac{1}{p^2} \big) \cdot \sum_{p}^{}\frac{\log p}{ p^2} (1 -  \frac{1}{p^2})^{-1} = 4\frac{1}{\zeta(2)} \frac{- \zeta^{\prime}(2)}{\zeta(2)} =  -4\frac{ \zeta^{\prime}(2)}{\zeta^2(2)} \,.$$
\end{lemma}

\begin{proof}
The proof is similar to the proof of Lemma \ref{Constan1}. We have
    \begin{align*}
&\sum_{m=1}^{\infty}\sum_{n=1}^{\infty} \frac{\mu(m) \cdot \mu(n)}{(mn)^b [m^a, n^a]} \log\frac{(m^a, n^a)}{[m^a, n^a]} \\= &a \sum_{k = 1}^{\infty}\lb \sum_{\alpha_k, \beta_k \in \{0,\, 1 \}} \frac{(-1)^{\alpha_k + \beta_k} \lb - |\alpha_k - \beta_k| \rb \log p_k}{p_k^{a \max \{ \alpha_k, \beta_k\} + b(\alpha_k + \beta_k) }} \rb \prod_{\substack{i = 1\\i \neq k}}^{\infty} \lb \sum_{\alpha_i, \beta_i \in \{0,\, 1 \}} \frac{(-1)^{\alpha_i + \beta_i}}{p_i^{a \max \{ \alpha_i, \beta_i\} + b(\alpha_i + \beta_i) }} \rb \\= &2a\sum_{k = 1}^{\infty}\lb \frac{\log p_k}{p_k^{a + b }} \rb \prod_{\substack{i = 1\\i \neq k}}^{\infty} \lb 1 - \frac{2}{p_i^{a + b}} + \frac{1}{p_i^{a + 2b}}\rb\\= &2a\sum_{k = 1}^{\infty}\lb \frac{\log p_k}{p_k^{a + b }} \rb \lb 1 - \frac{2}{p_k^{a + b}} + \frac{1}{p_k^{a + 2b}}\rb^{-1} \prod_{\substack{i = 1}}^{\infty} \lb 1 - \frac{2}{p_i^{a + b}} + \frac{1}{p_i^{a + 2b}}\rb.
    \end{align*}
  Take $b = 1 - \frac{a}{2}$ in the above formula, then we are done.
\end{proof}

\section{Proof of Theorem \ref{Main}}\label{proof}
Let $\beta \in (\frac{2}{3}, 1) $ and $\eta \in (0, \beta - \frac{2}{3})$ be fixed. Set  $X = \e^{\left(\log T\right)^{\beta}}$, $M  = X^a$, and  $b = 1 - \frac{a}{2}$. Define the arithmetic function $ \a: \N \to \R$ by 
$$
  \a(h) = \begin{cases}
			\frac{\mu(n)}{n^b} , & \text{if\, $h = n^a$ \text{for some} $n \in \N,$ }\\
   \\
           0, & \text{otherwise.}
		 \end{cases}
$$

and let  $\A(s) = \sum_{m \leqslant M} \a(m) m^{-s}.$ Then we have $$\A(\frac{1}{2} + \i t) = \sum_{n \leqslant M ^{\frac{1}{a}}}  \frac{\mu(n)}{n^{1+ \i at}}.$$

And we define  $E_a(t, T)$ by \begin{align*}
    E_a(t, T) &= \frac{1}{\zeta(1+\i a t)}  - \sum_{n \leqslant  \e^{\left(\log T\right)^{\beta}} }   \frac{\mu(n)}{n^{1+ \i at}} = \frac{1}{\zeta(1+\i at)} - \A\lb\frac{1}{2} + \i t\rb .
    \end{align*}

    We now split the integral $\int_T^{2T}|\zeta(\tfrac{1}{2} + \i t)/\zeta(1 + \i at)|^2 \d t$ into a sum of several integrals
\begin{align}\label{goal}
    &\int_{T}^{2T} \left|\frac{\zeta\left(\frac{1}{2}+\i t\right)}{\zeta\left(1+\i at\right)}\right|^2 \d t\\   = \nonumber & M_a(2T) - M_a(T) + \int_{T}^{2T} \left|\zeta\left(\frac{1}{2}+\i t\right)\right|^2 \left|E_a(t, T)\right|^2 \d t +  O\left(  \int_{T}^{2T} \left|\zeta\left(\frac{1}{2}+\i t\right)\right|^2 \left|\A\lb\frac{1}{2} + \i t\rb\right| \left| E_a(t, T)\right| \d t \right),
\end{align}

where $M_a(T)$ is defined as
\begin{align*}
    M_a(T) = \int_{0}^{T} \left|\zeta(\frac{1}{2}+\i t)\right|^2 \left|\A\lb\frac{1}{2} + \i t\rb\right|^2 \d t.
\end{align*}

Next, we use the theorem of of Balasubramanian--Conrey--Heath-Brown (Lemma \ref{BCHB}) to obtain
\begin{align*}
    M_a(T) & = T \sum_{h,\, k\, \leqslant X^a} \frac{\a(h)}{h}\frac{ \overline{\a(k)}  }{k} (h, k)\left( \log \frac{T(h, k)^2}{2 \pi h k} + 2\gamma -1\right) + O_B\left( T \lb\log T\rb^{-B} \right)\\
    & = S_1 + S_2 + S_3 + O_B\left( T \lb\log T\rb^{-B} \right)
\end{align*}

with
\begin{align*}
   S_1 & = T \log T \sum_{m,\,n \leqslant X 
 }\frac{\mu(m) \cdot \mu(n)}{(mn)^b [m^a, n^a]}, \\
     S_2 & = T (- \log 2\pi + 2\gamma -1)\sum_{m,\,n \leqslant X 
 }\frac{\mu(m) \cdot \mu(n)}{(mn)^b [m^a, n^a]},\\
      S_3 & = T \sum_{m,\,n \leqslant X 
 }\frac{\mu(m) \cdot \mu(n)}{(mn)^b [m^a, n^a]} \log\frac{(m^a, n^a)}{[m^a, n^a]}.\\
\end{align*}
 By Lemmas \ref{Constan1}, 
 \ref{Constan1: err},  \ref{Constan2: err}, \ref{Constan2}, we compute

 \begin{align*}
     \sum_{m,\,n \leqslant X 
 }\frac{\mu(m) \cdot \mu(n)}{(mn)^b [m^a, n^a]} = \D_1(a) + O\lb \frac{\log X}{ \sqrt X}\rb,
 \end{align*}
 
 \begin{align*}
      \sum_{m,\,n \leqslant X 
 }\frac{\mu(m) \cdot \mu(n)}{(mn)^b [m^a, n^a]} \log\frac{(m^a, n^a)}{[m^a, n^a]} = \widetilde{\D}_0(a) +   O\lb \frac{\log^2 X}{ \sqrt X}\rb.
 \end{align*}
 Recall that $X = \e^{\left(\log T\right)^{\beta}}$,  so we find that
\begin{align}\label{MaT}
    M_a(T) = \D_1(a) T\log T + \lb(- \log 2\pi + 2\gamma -1)\D_1(a) + \widetilde{\D}_0(a)\rb T + O_B\left( T \lb\log T\rb^{-B} \right).
\end{align} 

By the above formula \eqref{MaT}, we have $M_a(T) \sim \D_1 (a) T \log T $, so $(M_a(2T) - M_a(T))  \sim \D_1 (a) T \log T\ll_a T\log T $.  Thus by Cauchy-Schwarz, Hardy-Littlewood's formula \eqref{HL}, and Lemma \ref{approx}, we have
\begin{align*}
     &\int_{T}^{2T} \left|\zeta\left(\frac{1}{2}+\i t\right)\right|^2 \left|\A\lb\frac{1}{2} + \i t\rb\right| \left| E_a(t, T)\right| \d t \\ &\ll_{a,\, \beta,\, \eta}  \e^{-\left(\log T\right)^{\eta}} \sqrt{\int_{T}^{2T} \left|\zeta(\frac{1}{2}+\i t)\right|^2  \d t } \sqrt{\int_{T}^{2T} \left|\zeta\left(\frac{1}{2}+\i t\right)\right|^2 \left|\A\lb\frac{1}{2} + \i t\rb\right|^2\d t}\\ &\ll_{a,\, \beta,\, \eta} \e^{-\left(\log T\right)^{\eta}} T\log T.
\end{align*}

Again by Hardy-Littlewood’s formula \eqref{HL}  and Lemma \ref{approx}, we obtain
\begin{align*}
    \int_{T}^{2T} \left|\zeta\left(\frac{1}{2}+\i t\right)\right|^2 \left|E_a(t, T)\right|^2 \d t \ll_{a,\, \beta,\, \eta} T \lb \log T \rb \e^{-2\left(\log T\right)^{\eta}}.
\end{align*}

Theorem \ref{Main}  follows from \eqref{goal}, \eqref{MaT},   and the above two estimates.

%Return to  \eqref{goal}, then we are done.

\section{Further Discussions,  Problems and Conjectures}\label{sec:res}

Theorem \ref{Main}  holds for  $a \in \N$, but we believe
 that such a formula should exist for any positive number $a$ and we  aim to obtain a power-saving error term in the formula. We propose the following  problem.

\begin{problem}
    Let  $a \in (0, \infty)$ be fixed. Show that for all large $T$, the following holds
 \begin{align*}
\int_{T}^{2T} \left|\frac{\zeta\lb\frac{1}{2}+ \i t\rb}{\zeta\lb1+ \i at\rb}\right|^2 \d t = \D_1(a)\, T \log T + \D_0(a) \, T   + O(T^{1 -\theta}),
\end{align*}
for   some positive constant $\theta$, where $\D_1(a)$ and $\D_0(a)$ are the same as in Theorem \ref{Main}.
\end{problem}

The next problem is inspired by Corollary \ref{LimiOfa}.

\begin{problem}
    For what kind of function $ a = a(x)$ such that $ a\lb\R_{>0}\rb \subset \R_{>0}$ and $\lim_{T\to \infty} a(T) = \infty$, the following limit could exist? 
$$\lim_{T \to \infty}\frac{1}{T \log T} \int_{T}^{2T} \left|\frac{\zeta\lb\frac{1}{2}+\i t\rb}{\zeta\lb1+\i a\lb T\rb t\rb}\right|^2 \d t$$
And when does the limit could be equal to $\zeta(2)/\zeta(4)?$
\end{problem}

We propose the following conjecture, inspired  by Theorem \ref{Main}, \eqref{Ingham} and the  folklore  conjecture which states that: for  given $k>0$,  $\int_T^{2T}|\zeta(\tfrac{1}{2} + \i t)|^{2k}\d t \sim \CC_k  T \log^{k^2} T, $    as $T \to \infty$, for some positive constant $\CC_k$ (see section 4 and 5 in \cite{Sound} for further details and introduction). This conjecture has been refined and generalized to other families of $L$-functions in \cite{CFKRS}. In particular, it is conjectured that  $\int_T^{2T}|\zeta(\tfrac{1}{2} + \i t)|^{2k}\d t = T P_{k^2}(\log T) +  O_{\epsilon}(T^{\frac{1}{2}+\epsilon}) $, for some polynomial $P_{k^2}(\cdot)$ of degree $k^2$.

\begin{conjecture}
Let  $a, k\in (0, \infty)$, $ \beta \in \R$ and $\sigma \in (\frac{1}{2}, 1]$ be fixed. Then for all large $T$, we have
 \begin{align*}
\int_{T}^{2T} \left|\frac{\zeta\lb\frac{1}{2}+ \i t\rb}{\left|\zeta\lb\sigma + \i at\rb\right|^{\beta}}\right|^{2k} \d t  =  T P_{k^2}\lb \log T\rb +   O_{\sigma,\, a,\, k,\, \beta}\lb \frac{T}{\log T} \rb, %\quad \quad \quad  as\,\, T \to\infty,
\end{align*}
for some polynomial $P_{k^2}(\cdot)$ of degree $k^2$ which  only depends on $\sigma$, $a$, $k$ and $\beta $.
\end{conjecture}

Using the random matrix theory,   Keating and Snaith \cite{Keating-Snaith} successfully predicted  the constant $\CC_k$.We are interested in exploring whether the random matrix theory could also be useful in predicting the leading coefficients of the polynomial in the above conjecture.

We believe similar statements should hold true for other families of $L$-functions as well. For example, when considering mean values for ratios of Dirichlet 
$L$-functions, we propose the following two conjectures. For related conjectures on moments of 
$L$-functions, refer to  \cite{Sound}.

\begin{conjecture}
Let  $k\in (0, \infty)$, $ \beta \in \R$ and $\sigma \in (\frac{1}{2}, 1]$ be fixed. Then for all large prime $q$, we have
 \begin{align*}
\sideset{}{^\star}\sum_{\chi \, \lb\emph{mod}\, q\rb} \left|\frac{L\lb\frac{1}{2}, \chi \rb}{\left|L\lb\sigma , \chi \rb\right|^{\beta}}\right|^{2k}    = q P_{k^2}(\log q) +   O_{\sigma,\, k,\, \beta}\lb  \frac{q}{\log q} \rb, %\quad \quad \quad  as\,\,\, prime\,\,\, q \to \infty,
\end{align*}
for some polynomial $P_{k^2}(\cdot)$ of degree $k^2$  which only depends on $\sigma$, $k$ and $\beta $. Here,  the sum is over primitive characters $\chi$.
\end{conjecture}

\begin{conjecture}
Let  $k\in (0, \infty)$, $ \beta \in \R$ and $\sigma \in (\frac{1}{2}, 1]$ be fixed. Then for all large $X$, we have
 \begin{align*}
\sideset{}{^\flat}\sum_{|d|\leqslant X} \lb  \frac{L\lb\frac{1}{2}, \chi_d \rb}{L\lb\sigma , \chi_d \rb^{\beta}}\rb^{k} = X P_{\frac{k(k+1)}{2}}(\log X) + O_{\sigma,\, k,\, \beta}\lb  \frac{X}{\log X} \rb , %\quad \quad \quad  as\,\,\, \,\,\, X \to \infty,
\end{align*}
for some polynomial $P_{\frac{k(k+1)}{2}}(\cdot)$ of degree $\frac{k(k+1)}{2}$  which only depends on $\sigma$, $k$ and $\beta $. Here, the  sum is over fundamental discriminants $d$ and $\chi_d$ denotes the associated primitive
quadratic character.
\end{conjecture}

 \section*{Acknowledgements}
 I  thank  Kristian Seip for several helpful discussions. Most of the work was done when I was visiting the Norwegian University of Science and Technology, and most part of the paper was written when I was visiting Shandong University. I thank Kristian Seip and Yongxiao Lin for their hospitality.  The work was supported by the Austrian Science
Fund (FWF), project W1230.  I am currently supported  by a postdoctoral fellowship funded by  the Courtois Chair II in fundamental research;
 the Natural Sciences and Engineering Research Council of Canada; and
the Fonds de recherche du Qu\'ebec - Nature et technologies. Finally, I  thank the anonymous referee who read this paper carefully and provided valuable comments   that improved the presentation of
the paper.

%\subsection{transcendental}

\end{document}